\documentclass[11pt,a4paper,leqno]{article}
\usepackage{a4wide}
\usepackage{latexsym}
\usepackage{amsmath}
\usepackage{amssymb}
\usepackage{stackrel} 
\usepackage{color}
\usepackage{graphicx}

\usepackage{hyperref}

\newtheorem{lemma}{Lemma}

\newtheorem{theo}{Theorem}

\newenvironment{proof}{\medskip\par\noindent{\bf Proof}}{\hfill $\Box$
\medskip\par}

\newcommand{\C}{\mathbb{C}}
\newcommand{\R}{\mathbb{R}}
\newcommand{\N}{\mathbb{N}}
\newcommand{\Z}{\mathbb{Z}}
\newcommand{\I}{\rm i}

\begin{document}
\title{On the convergence of generalized power series solutions of $q$-difference equations}
\author{Renat Gontsov, Irina Goryuchkina, Alberto Lastra}
\date{}
\maketitle
\thispagestyle{empty}

\begin{abstract}
A sufficient condition for the convergence of a generalized formal power series solution to 
an algebraic $q$-difference equation is provided. The main result leans on a geometric property 
related to the semi-group of (complex) power exponents of such a series. This property corresponds 
to the situation in which the small divisors phenomenon does not arise. Some examples illustrating 
the cases where the obtained sufficient condition can be or cannot be applied are also depicted.
\end{abstract}

\medskip

\noindent {\bf MSC Class:} 39A13 (Primary); 39A45 (Secondary). 

\section{Introduction}

Compared to formal solutions of algebraic equations, which are Puiseux power series at most, algebraic ordinary {\it differential} equations (ODEs)
in general may have formal solutions in the form of power series with complex power exponents. First examples of such formal solutions, without a general approach to computing formal solutions of such a nature to algebraic ODEs though, can be found in papers studying Painlev\'e equations near their fixed singular points (see \cite{j,k,s,t}). Later, general computational methods were proposed by J.\,Della Dora, F.\,Richard-Jung \cite{ddrj} and A.\,D.\,Bruno \cite{b}, based on the generalization of the Newton--Puiseux polygonal method for algebraic equations. The first authors have developed the methods of D.\,Yu.\,Grigoriev, M.\,Singer \cite{gs} and J.\,Cano \cite{c} which, in turn, go back to the results of H.\,Fine \cite{f}. Bruno's algorithmic methods use 
the same principles as the papers by Della Dora, Richard-Jung and their predecessors, partially repeating the latter and extending as well. We may also note that complex power exponents in the asymptotic of actual solutions and their computation {\it via} the generalized Newton--Puiseux polygonal method already arise at the end of the nineteenth century in the works by M.\,Petrovitch \cite{p}. Power series with complex power exponents thus become a fairly familiar object in the theory of algebraic ODEs. Algebraic $q$-{\it difference} equations also may have such formal solutions, which are a quite new object for 
studying in the field though. A recent work by Ph.\,Barbe, J.\,Cano, P.\,Fortuny Ayuso, W.\,P.\,McCormick \cite{bcfm} dealing with formal solutions of algebraic $q$-difference equations in the form of Hahn series with {\it real} power exponents is one of the first contributions in this direction.

In our paper, we study a question of the convergence of formal power series solutions with complex power exponents to an algebraic $q$-difference equation. 
The corresponding problem in the {\it differential} case was treated by B.\,Malgrange \cite{m}, J.\,Cano \cite{c} for a formal {\it Taylor series} solution, 
who have proposed sufficient conditions of its convergence and, more generally, estimated the growth of the series coefficients in the case where it 
diverges (the Maillet--Malgrange theorem). Their results were generalized to formal power series solutions with {\it complex} power exponents in \cite{gg0}, \cite{gg}. As for the $q$-{\it difference} case, a question of convergence and, more generally, the Maillet--Malgrange type theorem have been studied
so far only for formal {\it Taylor series} solutions \cite{bcfm,b1,dv,lz,zh}. Even for such solutions, an additional phenomenon of small divisors may arise 
for $|q|=1$ (see \cite{b1,dv}), in contrast to the differential case where one does not meet this phenomenon studying a question of convergence.
For power series solutions with {\it complex} power exponents, the small divisors phenomenon may occur not only in the case of $|q|=1$. Beginning the study 
of such formal solutions to algebraic $q$-difference equations, we will restrict ourselves here to the case where small divisors do not arise.

We consider a $q$-difference equation 
\begin{equation}\label{e1}
F(z,y,\sigma y,\sigma^2 y,\ldots,\sigma^{n}y)=0,
\end{equation}
where $F=F(z,y_0,y_1,\ldots,y_n)$ is a polynomial and $\sigma$ stands for the dilatation operator 
$$
\sigma(f(z))=f(qz),
$$ 
$q\ne0,1$ being a fixed complex number.

For any sequence of complex numbers $(\lambda_j)_{j\geqslant0}$ that satisfy 
\begin{itemize}
\item[(i)] ${\rm Re}\,\lambda_j\geqslant 0$ for all $j\geqslant0$,
\item[(ii)] ${\rm Re}\,\lambda_j\leqslant\hbox{Re}\,\lambda_{j+1}$ for all $j\geqslant0$,
\item[(iii)] $\lim_{j\to\infty}{\rm Re}\,\lambda_j=+\infty$,
\end{itemize}
one may consider a formal power series with complex power exponents $\lambda_j$, $\sum_{j=0}^{\infty}c_jz^{\lambda_j}$, which will be called a {\it generalized} formal power series.

We note that the conditions (i), (ii), (iii) make the set of all generalized formal power series an algebra over $\C$. The definition of the dilatation operator extends naturally to this algebra after fixing the value $\ln q$ by the condition $0\leqslant{\rm arg}\,q<2\pi$:
$$
\sigma\Bigl(\sum_{j=0}^{\infty}c_jz^{\lambda_j}\Bigr)=\sum_{j=0}^{\infty}c_jq^{\lambda_j}z^{\lambda_j}.
$$
Thus the notion of a generalized formal power series solution of (\ref{e1}) is correctly defined in view of the above remarks: such a series $\varphi$ 
is said to be a {\it formal solution} of (\ref{e1}) if the substitution of $\varphi$ into the polynomial $F$ leads to a generalized power series with zero coefficients. We will start with an example of such a formal solution (Section 2), then give some auxiliary statements concerning generalized power series solutions of algebraic $q$-difference equations (Section 3) and propose a sufficient condition of the convergence of such solutions (Theorem 1 in Section 4). In the last section we give and discuss examples illustrating different situations concerning Theorem 1.
 
\section{Preliminary examples}

Let us start with an example of the differential equation Painlev\'e III with the parameters $a=b=0$, $c=d=1$:
$$
\frac{d^2y}{dz^2}=\frac1y\left(\frac{dy}{dz}\right)^2-\frac1z\frac{dy}{dz}+y^3+\frac1y.
$$
Rewritten with respect to the differential operator $\delta=z(d/dz)$ this becomes
\begin{equation}\label{P3}
y\,\delta^2 y-(\delta y)^2-z^2y^4-z^2=0. 
\end{equation}
This equation has a two-parameter family of generalized formal power series solutions \cite{Gr,s}:
\begin{equation}\label{P3sol}
\varphi=C\,z^r+\sum_{\lambda\in\rm K}c_{\lambda}z^{\lambda},
\end{equation}
where $C\ne0$ is an arbitrary complex number and $r$ is any complex number with $-1<{\rm Re}\,r<1$. 
The other coefficients $c_{\lambda}$ are uniquely determined by $C$, and the set $\rm K$ of power exponents is of the form
\begin{equation}\label{K}
{\rm K}=\{r+m_1(2-2r)+m_2(2+2r)\mid m_1,m_2\in{\mathbb Z}_+, m_1+m_2>0\}.
\end{equation}
This can be observed by making the change of variable $y=C\,z^r+z^r\,u$, under which (\ref{P3}) is
transformed to an equation of the form
\begin{equation}\label{P3tr}
C\,\delta^2u=z^{2-2r}+C^4\,z^{2+2r}+(\delta u)^2-u\,\delta^2u+z^{2+2r}u\,P_3(u),
\end{equation}
$P_3$ being a polynomial of the third degree. Thus, searching for its generalized power series solution
$u=\sum_{j=0}^{\infty}c_j\,z^{\lambda_j}$ for a given non-zero $\lambda_0$, one comes to the relation
$$
C\sum_{j=0}^{\infty}c_j\lambda_j^2\,z^{\lambda_j}=z^{2-2r}+C^4\,z^{2+2r}+(\delta u)^2-u\,\delta^2u+ z^{2+2r}u\,P_3(u).
$$
From this one deduces that $\lambda_0$ belongs to the additive semi-group generated by the numbers $2-2r$, $2+2r$ and finds $c_0$, further proceeding recursively with the other $\lambda_j$'s, $c_j$'s. Moreover, the obtained generalized power
series solution converges in sectors of small radius with the vertex at the origin and of opening less than $2\pi$ (see \cite{gg0}, \cite{s}). We also note that in the case of non-zero parameters $a$, $b$ of the Painlev\'e III equation, the set $\rm K$ of the power exponents of the solution (\ref{P3sol}) is more dense, with the generators $1-r$, $1+r$.

If we consider a $q$-difference analogue of the equation (\ref{P3}) formally changing $\delta$ by $\sigma$,
\begin{equation}\label{qP3}
y\,\sigma^2 y-(\sigma y)^2-z^2y^4-z^2=0, 
\end{equation}
we will see that the existence of a formal solution of the same form (\ref{P3sol}), (\ref{K}) as in the differential case, is guaranteed not for any $r$ with $-1<{\rm Re}\,r<1$. Indeed, after the change of the unknown $y=C\,z^r+z^r\,u$, one comes to a $q$-difference equation
$$
C\,(\sigma-1)^2u=q^{-2r}z^{2-2r}+C^4\,q^{-2r}z^{2+2r}+(\sigma u)^2-u\,\sigma^2u+z^{2+2r}u\,P_3(u),
$$ 
whose right-hand side is similar to that of (\ref{P3tr}). However, under the action of the left-hand side operator 
$C\,(\sigma-1)^2$ on a potential generalized formal power series solution $u=\sum_{j=0}^{\infty}c_j\,z^{\lambda_j}$,
with $\lambda_j$'s belonging to the additive semi-group generated by the numbers $2-2r$, $2+2r$, this $u$ turns into
$$
C\sum_{j=0}^{\infty}(q^{\lambda_j}-1)^2c_j\,z^{\lambda_j}.
$$
This means that one can guarantee the existence of a formal solution of the form (\ref{P3sol}), (\ref{K}) to the equation (\ref{qP3}), if none of the roots $2\pi{\rm i}k/\ln q$, $k\in\Z$, of the equation $q^{\lambda}=1$ belongs to that semi-group. In particular, this holds {\it if the numbers $1-r$, $1+r$, $\pi{\rm i}/\ln q$ are linearly independent over $\Z$ or if the numbers $1-r$, $1+r$ lie in the open half-plane in $\C$ on one side of the line with the slope $\ln|q|/{\rm arg}\,q$ passing through the origin.} As we will see further, in the second case the convergence of such a formal solution in sectors of small radius with the vertex at the origin and of opening less than $2\pi$ is guaranteed.

\section{Preliminary statements}

Let 
\begin{equation}\label{formsol}
\varphi=\sum_{j=0}^{\infty}c_jz^{\lambda_j}, \qquad c_j\in\C^*,
\end{equation} 
be a generalized formal power series solution of (\ref{e1}), that is,
$$
F(z, \Phi)=0, \qquad \Phi:=(\varphi,\sigma\varphi,\ldots,\sigma^{n}\varphi).
$$
\noindent\textbf{Assumption (A):} For each $k=0,\ldots,n$, the generalized formal power series $F'_{y_k}(z,\Phi)$ is 
of the form
$$
\frac{\partial F}{\partial y_k}(z,\Phi)=A_kz^{\lambda}+B_kz^{\tilde\lambda_k}+\ldots, \qquad {\rm Re}\,\tilde\lambda_k>{\rm Re}\,\lambda,
$$
$\lambda$ being the same for all $k=0,\ldots,n$ and at least one of the $A_k$'s being non-zero.  
\medskip

Under the above assumption, let us define a non-zero polynomial 
$$
L(\xi)=\sum_{k=0}^nA_k\,\xi^k
$$ 
of degree $\leqslant n$. Note that some of $A_k$'s may be equal to zero.
\medskip

\noindent{\bf Assumption (B):} There is $j_0\in\Z_+$ such that for the power exponents $\lambda_j$ of the formal solution (\ref{formsol}) of (\ref{e1}) satisfying Assumption (A), there holds
$$
L(q^{\lambda_j})\ne0, \quad \mbox{for all} \quad j\geqslant j_0.
$$

\noindent{\bf Remark 1.} In the case where all $\lambda_j$'s are real and $|q|\ne1$, this assumption holds automatically, 
since $\lim_{j\to\infty}q^{\lambda_j}=\infty$ (if $|q|>1$) or $\lim_{j\to\infty}q^{\lambda_j}=0$ (if $|q|<1$). 
\medskip

In this section we prove two lemmas which are auxiliary for deducing the main theorem on the convergence of the formal solution (\ref{formsol}). Let us preliminary fix the following notations.

For any $m\geqslant0$ we denote by $\varphi_m$ the truncation of $\varphi$ at the index $m$, i.\,e., 
$$
\varphi_m=\sum_{j=0}^mc_jz^{\lambda_j}, \qquad \Phi_m:=(\varphi_m,\sigma\varphi_m,\ldots,\sigma^{n}\varphi_m).
$$ 
Then $\varphi$ is represented in the form
$$
\varphi=\varphi_m+z^{\lambda_m}\psi, \qquad \psi=\sum_{j=1}^{\infty}c_{m+j}\,z^{\lambda_{m+j}-\lambda_m},
$$ 
and, respectively,
$$
\Phi=\Phi_m+z^{\lambda_m}\Psi, \qquad \Psi=(\psi_0,\psi_1,\ldots,\psi_n):=(\psi,(q^{\lambda_m}\sigma)\psi,\ldots,(q^{\lambda_m}\sigma)^n\psi).
$$

The first lemma guarantees that equation (\ref{e1}) can be reduced to a standard form which will be useful for further considerations.

\begin{lemma}\label{L1}
Under Assumption $(\rm A)$, for any $m\geqslant0$ such that 
$$
{\rm Re}\,\lambda_m>{\rm Re}\,\lambda \qquad and \qquad {\rm Re}\,\lambda_{m+1}>{\rm Re}\,\lambda_m,
$$ 
a transformation of the dependent variable  
\begin{equation}\label{e6}
y=\varphi_{m}+z^{\lambda_m}u
\end{equation}
transforms $(\ref{e1})$ into an equation of the form
\begin{equation}\label{e7}
L(q^{\lambda_m}\sigma)u=M(z,u,\sigma u,\ldots,\sigma^nu),
\end{equation}
where $M$ is a finite linear combination of monomials of the form
$$
z^{\alpha}u^{p_0}(\sigma u)^{p_1}\ldots(\sigma^nu)^{p_n}, \qquad \alpha\in\C, \quad {\rm Re}\,\alpha>0, \quad p_i\in\Z_+.
$$
\end{lemma}
\begin{proof}
Since for each $k=0,\ldots,n$ the Taylor formula yields  
$$
\frac{\partial F}{\partial y_k}(z,\Phi)-\frac{\partial F}{\partial y_k}(z,\Phi_m)=z^{\lambda_m}\sum_{l=0}^n\frac{\partial^2 F} 
{\partial y_k\partial y_l}(z,\Phi_m)\psi_l+\ldots,
$$
the assumption $\hbox{Re}\,\lambda_m>\hbox{Re}\,\lambda$ made for $m$ entails that $F'_{y_k}(z,\Phi_m)$ is of the form
\begin{equation}\label{e5}
\frac{\partial F}{\partial y_k}(z,\Phi_m)=A_kz^{\lambda}+\widehat B_kz^{\hat\lambda_k}+\ldots, \qquad
\hbox{Re}\,\hat\lambda_k>\hbox{Re}\,\lambda,
\end{equation}
that is, the restriction of the polynomial $F'_{y_k}$ on the truncation $\Phi_m$ begins with the same term as its restriction on the whole $\Phi$ 
(if $A_k\ne0$) or with the term whose exponent has the real part greater than ${\rm Re}\,\lambda$ (if $A_k=0$). 

Again, applying the Taylor formula to the relation $F(z,\Phi)=0$ we arrive at
\begin{eqnarray}\label{e4}
0&=&F(z,\Phi_m+z^{\lambda_m}\Psi)=F(z,\Phi_m)+z^{\lambda_m}\sum_{k=0}^n\frac{\partial F}{\partial y_k} (z,\Phi_m)\psi_k+\\ \nonumber
& &+\frac12z^{2\lambda_m}\sum_{k,l=0}^n\frac{\partial^2 F}{\partial y_k\partial y_l}(z,\Phi_m)\psi_k\psi_l+\ldots.
\end{eqnarray}
In view of (\ref{e5}), (\ref{e4}) and the assumptions on $m$, we get that the monomials of $F(z,\Phi_m)$ all have the exponents 
whose real part is larger than $\hbox{Re}(\lambda_m+\lambda)$. At this point, one divides (\ref{e4}) by $z^{\lambda_m+\lambda}$ and obtains the equality
$$
L(q^{\lambda_m}\sigma)\psi-N(z,\psi,(q^{\lambda_m}\sigma)\psi,\ldots,(q^{\lambda_m}\sigma)^n\psi)=0,
$$
where $N(z,u_0,u_1,\ldots,u_n)$ is given by a finite linear combination of monomials of the form 
$$
z^{\alpha}u_0^{p_0}u_1^{p_1}\ldots u_n^{p_n},\qquad \alpha\in\C,\quad\hbox{Re}\,\alpha>0, \quad p_i\in\Z_+.
$$
Therefore, (\ref{e6}) transforms the equation (\ref{e1}) into (\ref{e7}), whose formal solution is given by $u=\psi$.
\end{proof}

The second lemma describes the structure of the set of the power exponents $\lambda_j$ of the generalized formal power series (\ref{formsol}) (see also 
Th. 3.7.4 in \cite{bcfm} for Hahn series with real power exponents).

\begin{lemma} \label{L2}
Under Assumption $(\rm B)$, there is $m\geqslant0$ such that all the numbers $\lambda_{m+j}-\lambda_m$, $j>0$, belong to a finitely generated additive semi-group $\Gamma$.
\end{lemma}

\begin{proof} Apply Lemma \ref{L1} with $m\geqslant j_0$ and obtain the equation (\ref{e7}). Let $\Gamma$ be an additive semi-group generated by a (finite) 
set of power exponents $\alpha$ of the variable $z$ contained in the monomials of $M(z,u,\sigma u,\ldots,\sigma^nu)$. Denote by $\alpha_1,\ldots,\alpha_s$ 
the generators of this semi-group, that is,
$$
\Gamma=\Bigl\{m_1\alpha_1+\ldots+m_s\alpha_s\mid m_i\in\Z_+,\;\sum_{i=1}^sm_i>0\Bigr\}, \qquad {\rm Re}\,\alpha_i>0.
$$

Let us use the equality for the generalized formal power series $\psi=\sum_{j=1}^{\infty}c_{m+j}\,z^{\lambda_{m+j}-\lambda_m}$:
\begin{equation}\label{e3}
L(q^{\lambda_m}\sigma)\psi=M(z,\psi,\sigma\psi,\ldots,\sigma^n\psi).
\end{equation}
Since 
$$
L(q^{\lambda_m}\sigma)\psi=\sum_{j=1}^{\infty}L(q^{\lambda_{m+j}})c_{m+j}\,z^{\lambda_{m+j}-\lambda_m}, \qquad L(q^{\lambda_{m+j}})\ne0,
$$
the first term of the generalized formal power series on the left-hand side of (\ref{e3}) is 
$$
L(q^{\lambda_{m+1}})c_{m+1}\,z^{\lambda_{m+1}-\lambda_m}\ne0.
$$
On the other hand, the first term of the series on the right-hand side of (\ref{e3}) is some monomial $cz^{\alpha}$, since any monomial $cz^{\alpha}\psi^{p_0}(\sigma\psi)^{p_1}\ldots(\sigma^n\psi)^{p_n}$ with $\sum_{k=0}^np_k>0$ begins with 
a term whose power exponent has the real part greater than ${\rm Re}(\lambda_{m+1}-\lambda_m)$. Therefore,
$$
\lambda_{m+1}-\lambda_m=\alpha\in\Gamma.
$$ 
By similar reasoning, for each $j>1$ we have
$$
\lambda_{m+j}-\lambda_m=\tilde\alpha+k_1(\lambda_{m+1}-\lambda_m)+\ldots+k_{j-1}(\lambda_{m+j-1}-\lambda_m), \qquad \tilde\alpha\in\Gamma, \quad k_i\in\Z_+,
$$
which finishes the proof with the use of mathematical induction.
\end{proof}

Without the loss of generality, one may assume that the generators $\alpha_1,\ldots,\alpha_s$ of the semi-group $\Gamma$ 
are {\it linearly independent over $\Z$} (see Lemma 3 in \cite{gg0}). This important property of the generators will indeed be assumed and 
used further.

\section{A result on convergence}

In this section we present a sufficient condition of the convergence of the generalized formal power series solution (\ref{formsol}) of (\ref{e1}), 
for which Assumption (B) holds. As we have seen, such a solution can be represented in the form
$$
\varphi=\varphi_m+z^{\lambda_m}\psi,
$$  
and the power exponents of the generalized formal power series $\psi$ belong to a finitely generated additive semi-group
$\Gamma$.  Moreover, the generators of $\Gamma$ are linearly independent over $\Z$.

\begin{theo} 
Let all the generators $\alpha_1,\ldots,\alpha_s$ of the semi-group $\Gamma$ lie in the {\rm open} half-plane in $\C$ on one side of the line $\cal L$ 
with the slope $\ln |q|/\arg q$ passing through the origin. Let, additionally, in the case where they lie {\rm above} $\cal L$ the coefficient $A_0$
of the polynomial $L$ be non-zero, or in the case where they lie {\rm under} $\cal L$ the coefficient $A_n$ of $L$ be non-zero. Then the series $\psi$, 
and therefore the series $\varphi$, converges in sectors of small radius with the vertex at the origin and of opening less than $2\pi$.
\end{theo}

\noindent{\bf Remark 2.} The case $\ln |q|/\arg q=0/0$ is excluded, since $q\ne1$. The cases $\ln |q|/\arg q=-\infty$ and $\ln |q|/\arg q=+\infty$ 
correspond to a {\it positive real} $q$ ($0<q<1$ and $q>1$, respectively). In these cases the line $\cal L$ becomes the {\it imaginary axis} of the 
complex plane. Therefore, as the generators of the semi-group $\Gamma$ lie in the right open half-plane of $\C$, they are automatically placed 
''above $\cal L$'' for $0<q<1$ and ''under $\cal L$'' for $q>1$.
\medskip

\noindent{\bf Remark 3.} When all the generators of the semi-group $\Gamma$ are {\it real} (and hence lie on the positive real axis) they are placed either
\begin{itemize}
\item above $\cal L$ (if $\ln |q|/\arg q<0$ and hence $|q|<1$) or 
\item under $\cal L$ (if $\ln |q|/\arg q>0$ and hence $|q|>1$) or 
\item on $\cal L$ (if $\ln |q|/\arg q=0$ and hence $|q|=1$). 
\end{itemize}
Therefore, in the first two cases Theorem 1 becomes Theorem 6.1.1 from \cite{bcfm} or the particular case of convergence from Main Theorem in \cite{zh} 
(more precisely, a generalization of those theorems to the case of real generators, as they deal with formal {\it Taylor series} solutions of analytic $q$-difference equations). The third case, that of $|q|=1$, is exceptional and requires additional studying since the {\it small divisors} phenomenon 
arises (see \cite{b1} and \cite{dv}). On the other hand, this case has no specific feature if the generators of $\Gamma$ do not belong to $\R$.
Its analogue in the case of {\it complex} generators is a situation where they lie on the line $\cal L$ or on opposite sides of this line. We don't consider
here such a situation leading to the small divisors phenomenon and deserving a further study.      
\medskip

\noindent{\bf Remark 4.} The assumption of the theorem ``lie on one side of the line $\cal L$'' implies that all the numbers 
${\rm Re}\,\alpha_i\ln|q|-{\rm Im}\,\alpha_i\arg q$ are 
\begin{itemize}
\item negative, if the generators $\alpha_1,\ldots,\alpha_s$ lie {\it above} the line $\cal L$, and hence 
$$
k=\max\limits_{i=1,\ldots,s}({\rm Re}\,\alpha_i\ln|q|-{\rm Im}\,\alpha_i\arg q)<0;
$$ 
\item positive, if the generators $\alpha_1,\ldots,\alpha_s$ lie {\it under} the line $\cal L$, and hence 
$$
\tilde k=\min\limits_{i=1,\ldots,s}({\rm Re}\,\alpha_i\ln|q|-{\rm Im}\,\alpha_i\arg q)>0. 
$$
\end{itemize}
Therefore, when $\sum_{i=1}^sm_i$ tends to infinity one has
$$
\Bigl|q^{\sum_{i=1}^sm_i\alpha_i}\Bigr|=e^{\sum_{i=1}^sm_i({\rm Re}\,\alpha_i\ln|q|-{\rm Im}\,\alpha_i\arg q)}\leqslant 
e^{k\sum_{i=1}^sm_i}\rightarrow0 
$$
in the first case, and
$$
\Bigl|q^{\sum_{i=1}^sm_i\alpha_i}\Bigr|=e^{\sum_{i=1}^sm_i({\rm Re}\,\alpha_i\ln|q|-{\rm Im}\,\alpha_i\arg q)}\geqslant 
e^{\tilde k\sum_{i=1}^sm_i}\rightarrow\infty 
$$
in the second case.

\begin{proof} We consider the case where the generators of $\Gamma$ lie {\it above} the line $\cal L$ and $A_0\ne0$. The second case can be studied 
similarly. Further, for the simplicity of exposition, we will assume that $\Gamma$ is generated by two numbers $\alpha_1$, $\alpha_2$ linearly independent 
over $\Z$, 
$$
\Gamma=\Bigl\{m_1\alpha_1+m_2\alpha_2\mid m_1, m_2\in\Z_+,\;m_1+m_2>0\Bigr\}, \qquad {\rm Re}\,\alpha_1, {\rm Re}\,\alpha_2>0.
$$
In the case of an arbitrary number $s$ of generators all constructions are analogous, only multivariate Taylor series in $s$ rather than in two variables 
are involved.

We should establish the convergence of the generalized power series
$\psi=\sum_{j=1}^{\infty}c_{m+j}\,z^{\lambda_{m+j}-\lambda_m}$ which satisfies the equality (\ref{e3}) and whose
exponents $\lambda_{m+j}-\lambda_m$, $j>0$, are of the form 
$$
\lambda_{m+j}-\lambda_m=m_1\alpha_1+m_2\alpha_2, \qquad (m_1,m_2)\in Z,
$$
$Z$ being a uniquely determined subset of ${\mathbb Z}_+^2\setminus\{0\}$ such that the map $j\mapsto(m_1,m_2)$ is a bijection from $\N$ to $Z$. 
In other words,
$$
\psi=\sum_{(m_1,m_2)\in Z}c_{m_1,m_2}z^{m_1\alpha_1+m_2\alpha_2}=\sum_{(m_1,m_2)\in{\mathbb Z}_+^2\setminus\{0\}}c_{m_1,m_2}z^{m_1\alpha_1+m_2\alpha_2}
$$
(in the last series one puts $c_{m_1,m_2}=c_{m+j}$ for $(m_1,m_2)=(m_1(j),m_2(j))\in Z$, and $c_{m_1,m_2}=0$ for $(m_1,m_2)\not\in Z$). We will represent $\psi$ by a bivariate formal Taylor series and prove that the latter has a non-empty bidisc of convergence, whence the convergence of $\psi$ itself and 
Theorem 1 will follow. 

The linear independence of the generators $\alpha_1$, $\alpha_2$ over $\Z$ allows to define a bijective linear map 
$\iota:{\C}[[z^{\Gamma}]]\rightarrow{\C}[[z_1,z_2]]_*$ from the $\C$-algebra of generalized formal power series with exponents in $\Gamma$ to the 
$\C$-algebra of formal Taylor series in two variables without a constant term,
$$
\iota: \sum_{\gamma=m_1\alpha_1+m_2\alpha_2\in\Gamma}a_{\gamma}z^{\gamma}\mapsto
\sum_{(m_1,m_2)\in{\mathbb Z}_+^2\setminus\{0\}}a_{m_1,m_2}z_1^{m_1}z_2^{m_2}, \qquad a_{m_1,m_2}=a_{m_1\alpha_1+m_2\alpha_2}.
$$
Moreover, this map is an isomorphism of algebras, since
$$
\iota(\eta_1\eta_2)=\iota(\eta_1)\iota(\eta_2) \qquad\forall\eta_1,\eta_2\in{\C}[[z^{\Gamma}]].
$$
The $q$-difference operator $\sigma:{\C}[[z^{\Gamma}]]\rightarrow{\C}[[z^{\Gamma}]]$ naturally induces 
a linear automorphism $\tilde\sigma$ of ${\C}[[z_1,z_2]]_*$,
$$
\tilde\sigma: \sum_{(m_1,m_2)\in{\mathbb Z}_+^2\setminus\{0\}}a_{m_1,m_2}z_1^{m_1}z_2^{m_2}\mapsto
\sum_{(m_1,m_2)\in{\mathbb Z}_+^2\setminus\{0\}}q^{m_1\alpha_1+m_2\alpha_2}\,a_{m_1,m_2}z_1^{m_1}z_2^{m_2},
$$
which clearly satisfies $\tilde\sigma\circ\iota=\iota\circ\sigma$, so that the following commutative diagram holds:
$$
\begin{array}{ccc}
{\C}[[z^{\Gamma}]] & \stackrel{\sigma}{\longrightarrow} & {\C}[[z^{\Gamma}]] \\
\downarrow\lefteqn{\iota} & & \downarrow\lefteqn{\iota} \\
{\C}[[z_1,z_2]]_* & \stackrel{\tilde\sigma}{\longrightarrow} & {\C}[[z_1,z_2]]_*
\end{array}
$$

The reasoning above gives us the representation 
$$
\tilde\psi=\iota(\psi)=\sum_{(m_1,m_2)\in{\mathbb Z}_+^2\setminus\{0\}}c_{m_1,m_2}z_1^{m_1}z_2^{m_2}
$$ 
of the formal solution $\psi$ of \eqref{e7} by a bivariate formal Taylor series. Applying the map $\iota$ to the both sides of the equality 
\eqref{e3} we obtain the following relation for $\tilde\psi$:
\begin{equation}\label{psi0relation}
L(q^{\lambda_m}\tilde\sigma)\tilde\psi=\widetilde M(z_1,z_2,\tilde\psi,\tilde\sigma\tilde\psi,\ldots,
\tilde\sigma^n\tilde\psi),
\end{equation}
where $\widetilde M(z_1,z_2,u_0,\ldots,u_n)$ is a polynomial such that 
$$
\widetilde M(0,0,u_0,\ldots,u_n)\equiv0.
$$
Now using the relation (\ref{psi0relation}) we will prove that $\tilde\psi$ has a non-empty bidisc of convergence.

Let the polynomial $\widetilde M$ be written in the form
$$
\widetilde M(z_1,z_2,u_0,\ldots,u_n)=\sum_{k_1,k_2,{\bf p}}A_{k_1,k_2,{\bf p}}\,z_1^{k_1}z_2^{k_2}
u_0^{p_0}\ldots u_n^{p_n}, \qquad {\bf p}=(p_0,\ldots,p_n).
$$
To prove the convergence of $\tilde\psi\in{\mathbb C}[[z_1,z_2]]_*$ in some neighbourhood of the origin, we construct an equation
\begin{eqnarray}\label{majorant}
\nu\,W=\sum_{k_1,k_2,{\bf p}}|A_{k_1,k_2,{\bf p}}|\,z_1^{k_1}z_2^{k_2}W^{p_0}\ldots W^{p_n},
\end{eqnarray}
whose right-hand side is obtained from the polynomial $\widetilde M$ by the change of its coefficients $A_{k_1,k_2,{\bf p}}$ to their absolute values 
and all the $u_j$'s to the one variable $W$. The number $\nu$ is defined by the formula
$$
\nu=\inf_{j\geqslant1}|L(q^{\lambda_{j+m}})|=\inf_{(m_1,m_2)\in Z}|L(q^{\lambda_m+m_1\alpha_1+m_2\alpha_2})|,
$$
which is strictly positive, since $q^{m_1\alpha_1+m_2\alpha_2}\rightarrow0$ as $m_1+m_2\rightarrow\infty$ (see Remark 4) whereas $L(0)=A_0\ne0$.

The equation (\ref{majorant}) possesses a unique solution $W(z_1,z_2)$ holomorphic near the origin,
$$
W=\sum\limits_{(m_1,m_2)\in{\mathbb Z}_+^2\setminus\{0\}}C_{m_1,m_2}\,z_1^{m_1}z_2^{m_2},
$$
which satisfies the condition $W(0,0)=0$. This follows from the implicit function theorem. We prove that the power series 
$W$ is majorant for $\tilde\psi$:
$$
C_{m_1,m_2}\in{\mathbb R}_+, \qquad |c_{m_1,m_2}|\leqslant C_{m_1,m_2}\quad\forall (m_1,m_2)\in{\mathbb Z}_+^2 \setminus\{0\},
$$
which will imply the convergence of $\tilde\psi$ in a neighbourhood of the origin in $\C^2$.

First we use the equality (\ref{psi0relation}) to obtain recursive expressions for the coefficients $c_{m_1,m_2}$. Denote by $\phi$ the formal power 
series from the right-hand side of this equality,
$$
\phi=\sum_{k_1,k_2,{\bf p}}A_{k_1,k_2,{\bf p}}\,z_1^{k_1}z_2^{k_2}\,\tilde\psi^{p_0}(\tilde\sigma\tilde\psi)^{p_1}\ldots (\tilde\sigma^n\tilde\psi)^{p_n}\in{\mathbb C}[[z_1,z_2]]_*.
$$
Then (\ref{psi0relation}) implies
\begin{eqnarray}\label{clm}
L(q^{\lambda_m+m_1\alpha_1+m_2\alpha_2})\,c_{m_1,m_2}=\Bigl.\frac{\partial_1^{m_1}\partial_2^{m_2}\phi}
{m_1!\,m_2!}\Bigr|_{z_1=z_2=0},
\end{eqnarray}
where $\partial_1$, $\partial_2$ are the partial derivatives with respect to $z_1$, $z_2$. For the right-hand side of (\ref{clm}) we have
\begin{eqnarray}\label{form3}
\Bigl.\frac{\partial_1^{m_1}\partial_2^{m_2}\phi}{m_1!\,m_2!}\Bigr|_{z_1={z_2}=0}&=&
\sum_{k_1,k_2,{\bf p}}A_{k_1,k_2,{\bf p}}\times \nonumber \\
& &\times\sum_{\begin{array}{c}\scriptstyle l^{(1)}_0+\ldots+l^{(1)}_n=m_1-k_1\\ \scriptstyle l^{(2)}_0+\ldots+l^{(2)}_n= m_2-k_2\end{array}}
\Bigl.\frac{\partial_1^{l^{(1)}_0}\partial_2^{l^{(2)}_0}\tilde\psi^{p_0}}{l^{(1)}_0!\,l^{(2)}_0!}\,\;\ldots\;
\frac{\partial_1^{l^{(1)}_n}\partial_2^{l^{(2)}_n}(\tilde\sigma^n\tilde\psi)^{p_n}}{l^{(1)}_n!\,l^{(2)}_n!}\Bigr|_{z_1={z_2}=0}
\end{eqnarray}
(the internal sum is equal to zero, if $k_1>m_1$ or $k_2>m_2$), where
\begin{equation}\label{form4}
\Bigl.\frac{\partial_1^{l^{(1)}_j}\partial_2^{l^{(2)}_j}(\tilde\sigma^j\tilde\psi)^{p_j}}{l^{(1)}_j!\,l^{(2)}_j!}\Bigr|_{z_1={z_2}=0}=
q^{j\bigl(l^{(1)}_j\alpha_1+l^{(2)}_j\alpha_2\bigr)}\sum_{\begin{array}{c} \scriptstyle \lambda_1+\ldots+\lambda_{p_j}=l^{(1)}_j\\ \scriptstyle \mu_1+\ldots+\mu_{p_j}=l^{(2)}_j\end{array}}c_{\lambda_1,\mu_1}\ldots c_{\lambda_{p_j},\mu_{p_j}}.
\end{equation}
The summations in the formulae (\ref{form3}), (\ref{form4}) go over {\it non-negative} integers $l^{(1)}_j$, $l^{(2)}_j$ and $\lambda_i$, $\mu_i$. 
These formulae are similar to the formulae (18), (19) in \cite{gg0} for the differential case and follow from the Leibniz rule of
differentiation (for more technical details see \cite{gg0}). Since $\widetilde M(0,0,u_0,\ldots,u_n)\equiv0$, for each
coefficient $A_{k_1,k_2,{\bf p}}$ in (\ref{form3}) at least one of the indices $k_1$, $k_2$ is non-zero. Therefore, for
each fixed triple $(k_1,k_2,{\bf p})$ either all the summation indices $l^{(1)}_j$ are less than $m_1$ or
all the summation indices $l^{(2)}_j$ are less than $m_2$. Hence, the formula (\ref{clm}) can be written in the form
$$
L(q^{\lambda_m+m_1\alpha_1+m_2\alpha_2})\,c_{m_1,m_2}=p_{m_1,m_2}(\{A_{k_1,k_2,{\bf p}}\},\{c_{\lambda,\mu}\}),
$$
where $p_{m_1,m_2}$ is a polynomial of the variables $\{A_{k_1,k_2,{\bf p}}\}$, $\{c_{\lambda,\mu}\}$, with 
$\lambda\leqslant m_1$, $\mu\leqslant m_2$, and $\lambda+\mu<m_1+m_2$ (which is determined by the formulae (\ref{form3}), (\ref{form4})).

Now we similarly use the equality (\ref{majorant}) to obtain recursive expressions for the coefficients $C_{m_1,m_2}$ of the series $W=\sum_{(m_1,m_2)\in{\mathbb Z}_+^2\setminus\{0\}}C_{m_1,m_2}z_1^{m_1}z_2^{m_2}$ convergent near the origin in $\C^2$. Denote by $\Theta$ the power 
series from the right-hand side of this equality, 
$$
\Theta=\sum_{k_1,k_2,{\bf p}}|A_{k_1,k_2,{\bf p}}|\,z_1^{k_1}z_2^{k_2}\,W^{p_0}\ldots W^{p_n}
\in{\mathbb C}\{z_1,z_2\}.
$$
Then (\ref{majorant}) implies
\begin{eqnarray}\label{Clm}
\nu\,C_{m_1,m_2}=\Bigl.\frac{\partial_1^{m_1}\partial_2^{m_2}\Theta}{m_1!\,m_2!}\Bigr|_{z_1={z_2}=0},
\end{eqnarray}
and by analogy with $\phi$ one has
\begin{eqnarray}\label{form5}
\Bigl.\frac{\partial_1^{m_1}\partial_2^{m_2}\Theta}{m_1!\,m_2!}\Bigr|_{z_1={z_2}=0}&=&
\sum_{k_1,k_2,{\bf p}}|A_{k_1,k_2,{\bf p}}|\times \nonumber \\
& &\times\sum_{\begin{array}{c}\scriptstyle l^{(1)}_0+\ldots+l^{(1)}_n=m_1-k_1\\ \scriptstyle l^{(2)}_0+\ldots+l^{(2)}_n= m_2-k_2\end{array}}
\Bigl.\frac{\partial_1^{l^{(1)}_0}\partial_2^{l^{(2)}_0}W^{p_0}}{l^{(1)}_0!\,l^{(2)}_0!}\,\;\ldots\;
\frac{\partial_1^{l^{(1)}_n}\partial_2^{l^{(2)}_n}W^{p_n}}{l^{(1)}_n!\,l^{(2)}_n!}\Bigr|_{z_1={z_2}=0},
\end{eqnarray}
where
\begin{eqnarray}\label{form6}
\Bigl.\frac{\partial_1^{l^{(1)}_j}\partial_2^{l^{(2)}_j}W^{p_j}}{l^{(1)}_j!\,l^{(2)}_j!}\Bigr|_{z_1={z_2}=0}= 
\sum_{\begin{array}{c} \scriptstyle \lambda_1+\ldots+\lambda_{p_j}=l^{(1)}_j\\ \scriptstyle \mu_1+\ldots+\mu_{p_j}=l^{(2)}_j\end{array}}C_{\lambda_1,\mu_1}\ldots C_{\lambda_{p_j},\mu_{p_j}}.
\end{eqnarray}
Thus, the formula (\ref{Clm}) can be written in the form
\begin{eqnarray}\label{Plm}
\nu\,C_{m_1,m_2}=P_{m_1,m_2}(\{|A_{k_1,k_2,{\bf p}}|\},\{C_{\lambda,\mu}\}),
\end{eqnarray}
where $P_{m_1,m_2}$ is a polynomial of the variables $\{|A_{k_1,k_2,{\bf p}}|\}$, $\{C_{\lambda,\mu}\}$, with $\lambda\leqslant m_1$, 
$\mu\leqslant m_2$, and $\lambda+\mu<m_1+m_2$ (which is determined by the formulae (\ref{form5}), (\ref{form6}) and has real positive coefficients). 
Since for $(m_1,m_2)$ equal to $(1,0)$ and $(0,1)$ we have, respectively,
$$
\nu\,C_{1,0}=\partial_1\Theta(0,0)=|A_{1,0,{\bf 0}}|, \qquad \nu\,C_{0,1} =\partial_2\Theta(0,0)=|A_{0,1,{\bf 0}}|,
$$
all the coefficients $C_{m_1,m_2}$ are real non-negative numbers.

Finally we come to a concluding part of the proof, the estimates
$$
|c_{m_1,m_2}|\leqslant C_{m_1,m_2}\quad\forall (m_1,m_2)\in{\mathbb Z}_+^2\setminus\{0\}.
$$
We prove them by the induction with respect to the sum $m_1+m_2$ of the indices.

For $m_1+m_2=1$ according to (\ref{clm}) we have
$$
L(q^{\lambda_m+\alpha_1})\,c_{1,0}=\partial_1\phi(0,0)=A_{1,0,{\bf 0}}, \qquad L(q^{\lambda_m+\alpha_2})\,c_{0,1}=\partial_2\phi(0,0)=A_{0,1,{\bf 0}},
$$
hence 
\begin{eqnarray*}
|c_{1,0}|=\frac{|A_{1,0,{\bf 0}}|}{|L(q^{\lambda_m+\alpha_1})|}=\frac{\nu\,C_{1,0}}{|L(q^{\lambda_m+\alpha_1})|} \leqslant C_{1,0}, \\
|c_{0,1}|=\frac{|A_{0,1,{\bf 0}}|}{|L(q^{\lambda_m+\alpha_2})|}=\frac{\nu\,C_{0,1}}{|L(q^{\lambda_m+\alpha_2})|} \leqslant C_{0,1}
\end{eqnarray*}
(for $(1,0)\in Z$ we have $|L(q^{\lambda_m+\alpha_1})|\geqslant\nu$, whereas $c_{1,0}=0$ for $(1,0)\not\in Z$; the analogous situation remains valid 
for the index $(0,1)$).

Further, by the definition of the polynomials $p_{m_1,m_2}$ and $P_{m_1,m_2}$, for any $(m_1,m_2)\in Z$ we have
$$
\bigl|p_{m_1,m_2}(\{A_{k_1,k_2,{\bf p}}\},\{c_{\lambda,\mu}\})\bigr|\leqslant P_{m_1,m_2}(\{|A_{k_1,k_2,{\bf p}}|\},\{|c_{\lambda,\mu}|\}),
$$
since the estimate $|q^{l_1\alpha_1+l_2\alpha_2}|\leqslant1$ holds for any non-negative integers $l_1,l_2$ (see Remark 4). Therefore, the inductive 
assumption (the second inequality below) implies
\begin{eqnarray*}
|L(q^{\lambda_m+m_1\alpha_1+m_2\alpha_2})|\,|c_{m_1,m_2}|&=&\bigl|p_{m_1,m_2}(\{A_{k_1,k_2,{\bf p}}\},\{c_{\lambda,\mu}\})\bigr|\leqslant \\ & \leqslant & P_{m_1,m_2}(\{|A_{k_1,k_2,{\bf p}}|\},\{|c_{\lambda,\mu}|\})
\leqslant \\ & \leqslant & P_{m_1,m_2}(\{|A_{k_1,k_2,{\bf p}}|\},\{C_{\lambda,\mu}\})=\nu\,C_{m_1,m_2},
\end{eqnarray*}
whence the required estimate follows:
$$
|c_{m_1,m_2}|\leqslant\frac{\nu\,C_{m_1,m_2}}{|L(q^{\lambda_m+m_1\alpha_1+m_2\alpha_2})|}\leqslant C_{m_1,m_2}
$$
(again, for $(m_1,m_2)\in Z$ we have $|L(q^{\lambda_m+m_1\alpha_1+m_2\alpha_2})|\geqslant\nu$, whereas $c_{m_1,m_2}=0$ for $(m_1,m_2)\not\in Z$).

Now it remains to notice that for any sector $S\subset\C$ with the vertex at the origin and of opening less than $2\pi$, the terms of a generalized power 
series are holomorphic single-valued functions in $S$, and to pass from the convergence of $\tilde\psi=\sum_{(m_1,m_2)\in{\mathbb Z}_+^2\setminus\{0\}} c_{m_1,m_2}z_1^{m_1}z_2^{m_2}$ to the convergence of $\psi=\sum_{j=1}^{\infty}c_{m+j}z^{\lambda_{m+j}-\lambda_m}$ and $\varphi=\sum_{j=0}^{\infty}c_jz^{\lambda_j}$. Indeed, let the power series $\tilde\psi$ converge in a bidisc $\{|z_1|<r,\,|z_2|<r\}$. 
If $z\in S$ is small enough for the inequalities
$$
|z^{\alpha_1}|=|z|^{{\rm Re}\,\alpha_1}\,e^{-{\rm Im}\,\alpha_1\arg z}<r, \qquad 
|z^{\alpha_2}|=|z|^{{\rm Re}\,\alpha_2}\,e^{-{\rm Im}\,\alpha_2\arg z}<r
$$
to be held (recall that ${\rm Re}\,\alpha_{1,2}>0$), then
$$
|c_{m+j}z^{\lambda_{m+j}-\lambda_m}|=|c_{m_1,m_2}|\cdot|z^{\alpha_1}|^{m_1}\cdot|z^{\alpha_2}|^{m_2}<|c_{m_1,m_2}|\,r^{m_1}\,r^{m_2}, 
$$
hence the series $\psi=\sum_{j=1}^{\infty}c_{m+j}z^{\lambda_{m+j}-\lambda_m}$ and $\varphi=\sum_{j=0}^{\infty}c_jz^{\lambda_j}$ converge uniformly 
in $S$ for sufficiently small $|z|$. 

\end{proof}

\section{Concluding examples}

We recall that one of the $q$-analogues of Euler equation, 
$$
\sigma y-zy-z=0,
$$
possesses a formal Taylor series solution $\varphi=\sum_{m=1}^{\infty}q^{-m(m+1)/2}\,z^m$, and with null radius of convergence, if $|q|<1$. Let us
give an example of an algebraic $q$-difference equation and its divergent generalized power series solution whose power exponents form a semi-group
generated by {\it two} numbers.    
\medskip

{\bf Example 1.} Consider the algebraic $q$-difference equation 
\begin{equation}\label{ex1}
\sigma^2y-2^{-\sqrt2}\,\sigma y+y^2-5z=0,\qquad q=\frac12,
\end{equation}
which has a generalized formal power series solution 
\begin{equation}\label{ex2}
\varphi=\sum\limits_{m_1+m_2\geqslant 0}c_{m_1,m_2}\,z^{1+m_1+m_2(\sqrt{2}-1)},\qquad c_{m_1,m_2}\in\C,
\end{equation}
where $c_{01}$ is an arbitrary constant and the other coefficients are determined uniquely. This formal solution does not satisfy the assumptions of 
Theorem 1, since the exponents $1, \sqrt2-1$ lie above the line $\cal L$ but the polynomial $L(\xi)=\xi^2-2^{-\sqrt2}\xi$ vanishes at $\xi=0$. Let us 
prove that the series \eqref{ex2} has a zero radius of convergence. Substituting it in the equation \eqref{ex1} we obtain the equality
\begin{equation}\label{ex3}
\sum\limits_{m_1+m_2\geqslant0}c_{m_1,m_2}\Bigl(q^{2(1+m_1+m_2(\sqrt2-1))}-2^{-\sqrt2}\,q^{1+m_1+m_2(\sqrt{2}-1)}\Bigr)z^{1+m_1+m_2(\sqrt2-1)}=
\end{equation}
$$
=5z-\Bigl(\sum\limits_{m_1+m_2\geqslant 0}c_{m_1,m_2}z^{1+m_1+m_2(\sqrt2-1)}\Bigr)^2.
$$
Put 
$$z_1=z,\qquad z_2=z^{\sqrt{2}-1},
$$
then the equality \eqref{ex3} can be rewritten in the form
\begin{equation}\label{ex4}
\sum\limits_{m_1+m_2\geqslant0}c_{m_1,m_2}\Bigl(2^{-\sqrt2}\,q^{1+m_1+m_2(\sqrt2-1)}-q^{2(1+m_1+m_2(\sqrt2-1))}\Bigr)z_1^{m_1}z_2^{m_2}=
\end{equation}
$$
=-5+z_1\Bigl(\sum\limits_{m_1+m_2\geqslant0}c_{m_1,m_2}z_1^{m_1}z_2^{m_2}\Bigr)^2,
$$
whereas 
\begin{equation}\label{ex7}
\varphi=\sum\limits_{m_1+m_2\geqslant0}c_{m_1,m_2}z_1^{m_1+1}z_2^{m_2}.
\end{equation}
Then it follows from \eqref{ex4} that 
$$
c_{00}=\frac5{q^2-2^{-\sqrt{2}}\,q}=\frac{10}{2^{-1}-2^{-\sqrt2}}>80 
$$
and
\begin{equation*}
\sum\limits_{m_1+m_2\geqslant1}c_{m_1,m_2}\Bigl(2^{-\sqrt2}\,q^{1+m_1+m_2(\sqrt2-1)}-q^{2(1+m_1+m_2(\sqrt2-1))}\Bigr)z_1^{m_1}z_2^{m_2}=
\end{equation*}
$$
=z_1\Bigl(c_{00}+\sum\limits_{m_1+m_2\geqslant1}c_{m_1,m_2}z_1^{m_1}z_2^{m_2}\Bigr)^2.
$$
Hence $c_{01}$ is a free parameter and $c_{0,m_2}=0$ for all $m_2\geqslant2$. Further,
$$
c_{10}=\frac{c_{00}^2}{2^{-\sqrt2}\,q^2-q^4}=\frac{4c_{00}^2}{2^{-\sqrt2}-2^{-2}}
$$
and 
\begin{equation*}
\sum\limits_{m_1+m_2\geqslant2}c_{m_1,m_2}\,q^{1+m_1+m_2(\sqrt2-1)}\Bigl(q^{\sqrt2}-q^{1+m_1+m_2(\sqrt2-1)}\Bigr)z_1^{m_1}z_2^{m_2}=
\end{equation*}
$$
=2c_{00}z_1\sum\limits_{m_1+m_2\geqslant1}c_{m_1,m_2}z_1^{m_1}z_2^{m_2}+z_1\Bigl(\sum\limits_{m_1+m_2\geqslant1}c_{m_1,m_2}z_1^{m_1}z_2^{m_2}\Bigr)^2.
$$
If we take $c_{01}$ an arbitrary {\it positive} real number, the last equality will imply that all the coefficients $c_{m_1,m_2}$ are non-negative real 
numbers and
$$
c_{m_1,m_2}\,q^{1+m_1+m_2(\sqrt2-1)}\Bigl(q^{\sqrt2}-q^{1+m_1+m_2(\sqrt2-1)}\Bigr)\geqslant2c_{00}c_{m_1-1,m_2}, \qquad m_1\geqslant1,
$$
whence
$$
c_{m_1,0}\,q^{1+m_1}\Bigl(q^{\sqrt2}-q^{1+m_1}\Bigr)\geqslant2c_{00}c_{m_1-1,0}, \qquad m_1\geqslant1.
$$
Therefore, 
$$
c_{m_1,0}>c_{m_1-1,0}\,q^{-m_1}>c_{m_1-2,0}\,q^{-(m_1-1)-m_1}>{\dots}>c_{00}\,q^{-(m_1+1)m_1/2}.
$$
It follows that the series $\sum_{m_1=1}^{\infty}c_{m_1,0}z_1^{m_1}$ has a zero radius of convergence, hence the series \eqref{ex7} diverges for any 
positive real $z_1$, $z_2$.
\medskip

In the next example of a second order $q$-difference equation possessing a generalized power series solution whose power exponents form a semi-group
generated by {\it two} numbers again, these generators lie on the opposite sides of the line $\cal L$, though $|q|>1$ and the coefficient $A_2$ of
the polynomial $L(\xi)$ is non-zero. Thus one of the assumptions of Theorem 1 is not satisfied again and it seems that one also has divergence here.
\medskip  

{\bf Example 2.} The algebraic $q$-difference equation 
\begin{equation}\label{ex31}
\sigma^2 y-q^{1+\I}\sigma y=y^2+z^2,\qquad q=1.0001\,\I,\qquad \I=\sqrt{-1},
\end{equation}
possesses a generalized formal power series solution 
\begin{equation}\label{ex32}
\varphi=\sum\limits_{m_1+m_2\geqslant 0}c_{m_1,m_2}\,z^{(m_1+1)(1+\I)+m_2(1-\I)}, \qquad c_{m_1,m_2}\in\C,
\end{equation}
where $c_{00}$ is an arbitrary constant and the other coefficients are determined uniquely. This formal solution does also not satisfy 
the assumptions of Theorem 1, since the line $\cal L$ is almost horizontal ($\ln|q|/{\rm arg}\,q\approx0.00006<<1$) and the generators 
$1-\I$, $1+\I$ of the semi-group $\Gamma$ thus lie on the opposite sides of this line.
Note that, at the same time, $|q|>1$ and the coefficient $A_2$ of the polynomial $L(\xi)=\xi(\xi-q^{1+\I})$ is non-zero.

Similarly to the previous example, define 
$$
z_1=z^{1+\I}, \qquad z_2=z^{1-\I}.
$$ 
Then 
\begin{eqnarray*}
	\varphi&=&z_1\sum\limits_{m_1+m_2\geqslant 0}c_{m_1,m_2}\,z_1^{m_1}z_2^{m_2}, \\
	\sigma\varphi&=&z_1\sum\limits_{m_1+m_2\geqslant 0}q^{(m_1+1)(1+\I)+m_2(1-\I)}\,c_{m_1,m_2}\,z_1^{m_1}z_2^{m_2}, \\
	\sigma^2\varphi&=&z_1\sum\limits_{m_1+m_2\geqslant 0}q^{2(m_1+1)(1+\I)+2m_2(1-\I)}\,c_{m_1,m_2}\,z_1^{m_1}z_2^{m_2}.
\end{eqnarray*}
Substituting the above bivariate formal Taylor series into the equation \eqref{ex31} one comes to the equality
\begin{equation*}
\sum\limits_{m_1+m_2\geqslant 0}q^{(m_1+2)(1+\I)+m_2(1-\I)}\left(q^{m_1(1+\I)+m_2(1-\I)}-1\right)c_{m_1,m_2}\,z_1^{m_1}z_2^{m_2}=
\end{equation*}
$$
=z_1\Bigl(\sum\limits_{m_1+m_2\geqslant 0}c_{m_1,m_2}\,z_1^{m_1}z_2^{m_2}\Bigr)^2+z_2.
$$
Setting $m_2=0$ in it we will have 
\begin{equation*}
\sum\limits_{m_1=0}^{\infty}q^{(m_1+2)(1+\I)}\left(q^{m_1(1+\I)}-1\right)c_{m_1,0}\,z_1^{m_1}
=z_1\Bigl(\sum\limits_{m_1=0}^{\infty}c_{m_1,0}\,z_1^{m_1}\Bigr)^2,
\end{equation*}
and whence deduce that $c_{00}\in\C$ is a free parameter,
$$
c_{10}=\frac{c_{00}^2}{q^{3(1+\I)}(q^{1+\I}-1)}, \qquad c_{m_1,0}=\frac{2c_{00}\,c_{m_1-1,0}+ \sum_{l=1}^{m_1-2}c_{l,0}\,c_{m_1-1-l,0}}{q^{(m_1+2)(1+\I)}\left(q^{m_1(1+\I)}-1\right)}, \quad m_1>1.
$$
We note that $|q^{1+\I}|\approx0.2$ and the denominators in the above recurrent formulae for the coefficients descent to zero exponentially as $m_1$ 
tends to infinity. The numerical evidences indicate that the series $\sum_{m_1\geqslant0}|c_{m_1,0}|\,z_1^{m_1}$, with $c_{00}=1$, is rapidly diverging: 
$$
|c_{m_1,0}|>\bigl|q^{1+\I}\bigr|^{-(m_1+2)(m_1+3)/2}>5^{(m_1+2)(m_1+3)/2},
$$
beginning with $m_1=8$, which allows to suppose that the corresponding series (\ref{ex32}) has a zero radius of convergence.
\medskip

{\bf Example 3.} Let us return to the $q$-difference equation \eqref{qP3} from Section 2,
$$
y\,\sigma^2 y-(\sigma y)^2-z^2y^4-z^2=0.
$$  
Under an assumption that the numbers $1-r$, $1+r$, $\pi{\rm i}/\ln q$ are linearly independent over $\Z$, its generalized power series solution 
\eqref{P3sol} is of the form $\varphi=C\,z^r+z^r\,\psi$, where the power exponents of the generalized power series $\psi$ belong to the additive 
semi-group generated by $2-2r$, $2+2r$. In the case where $|q|=1$ we cannot apply Theorem 1 for studying the convergence of such a formal solution,
since the line $\cal L$ coincides with the real axis in this case and the numbers $2-2r$, $2+2r$ lie on different sides of this line. On the other hand, 
for any $q\in\C^*$ with $|q|\ne1$, the line $\cal L$ has a non-zero slope, hence there is an open set $U_q$ of values of the parameter $r$ for which 
the numbers $2-2r$, $2+2r$ lie on one side of this line. Since the polynomial $L$ in this example is $L(\xi)=C(\xi-1)^2$, both of its coefficients $A_0$,
$A_2$ are non-zero and, by Theorem 1, for $r\in U_q$ the series $\varphi$ converges in small sectors $S\subset\C$ with the vertex at the origin.   
\medskip

{\bf Acknowledgment.} A.\,Lastra is partially supported by the project PID2019-105621GB-I00 of Ministerio
de Ciencia e Innovaci\'on, Spain, and by Direcci\'on General de Investigaci\'on e 
Innovaci\'on, Consejer\'ia de Educaci\'on e Investigaci\'on of the Comunidad de Madrid 
(Spain), and Universidad de Alcal\'a under grant CM/JIN/2019-010, Proyectos de I+D 
para J\'ovenes Investigadores de la Universidad de Alcal\'a 2019.


\begin{thebibliography}{99}

\bibitem{bcfm}
Ph.\,Barbe, J.\,Cano, P.\,Fortuny Ayuso, W.\,P.\,McCormick, \emph{$q$-Algebraic equations, their power series solutions and the asymptotic behavior of 
their coefficients}, arXiv: 2006.09527 [math.CA] (2020).

\bibitem{b1}
J.-P.\,B\'ezivin, \emph{Sur les \'equations fonctionnelles aux $q$-diff\'erences}, Aequat. Math. \textbf{43} (1992), 159--176.

\bibitem{b2}
J.-P.\,B\'ezivin, \emph{Convergence des solutions formelles de certaines \'equations fonctionnelles}, Aequat. Math. \textbf{44} (1992), 84--99.

\bibitem{b}
A.\,D.\,Bruno, \emph{Asymptotic behaviour and expansions of solutions of an ordinary differential equation}, Russian Math. Surv.
\textbf{59}:3 (2004), 429--480.

\bibitem{c}
J.\,Cano, \emph{On the series defined by differential equations, with an extension of the Puiseux polygon construction to these
equations}, Analysis \textbf{13} (1993), 103--119.

\bibitem{ddrj} 
J.\,Della Dora, F.\,Richard-Jung, \emph{About the Newton algorithm for non-linear ordinary differential equations}, Proceedings of the 1997 
International symposium on Symbolic and algebraic computation, United States, 298--304.

\bibitem{dv}
L.\,Di Vizio, \emph{An ultrametric version of the Maillet--Malgrange theorem for nonlinear $q$-difference equations}, Proc. Amer. Math. Soc. 
\textbf{136}:8 (2008), 2803--2814.

\bibitem{f}
H.\,Fine, \emph{On the functions defined by differential equations, with an extension of the Puiseux polygon construction to these equations}, 
Amer. J. Math. {\bf 11}:4 (1889), 317--328.

\bibitem{gg0} 
R.\,R.\,Gontsov, I.\,V.\,Goryuchkina, \emph{On the convergence of generalized power series satisfying an algebraic ODE}, Asympt. Anal. \textbf{93}:4 
(2015), 311--325. 

\bibitem{gg} 
R.\,R.\,Gontsov, I.\,V.\,Goryuchkina, \emph{The Maillet--Malgrange type theorem for generalized power series}, Manuscr. Math. \textbf{156}:1-2 (2018), 
171--185.

\bibitem{Gr} A.\,V.\,Gridnev, \emph{Power expansions of solutions to the modified third Painlev\'e equation in a neighborhood of zero}, J. Math. Sci. \textbf{145} (2007), 5180--5187.

\bibitem{gs}
D.\,Yu.\,Grigor'ev, M.\,F.\,Singer, \emph{Solving ordinary differential equations in terms of series with real exponents}, Trans. Amer. Math. Soc. 
\textbf{327}:1 (1991), 329--351.

\bibitem{j} 
M.\,Jimbo, \emph{Monodromy problem and the boundary condition for some Painlev\'e equations}, Publ. RIMS Kyoto Univ. \textbf{18} (1982), 1137--1161.

\bibitem{k} 
H.\,Kimura, \emph{The construction of a general solution of a Hamiltonian system with regular type singularity and its application to Painlev\'e 
equations}, Ann. Mat. Pura Appl. \textbf{134} (1983), 363--392.

\bibitem{lz}
X.\,Li, C.\,Zhang, \emph{Existence of analytic solutions to analytic nonlinear $q$-difference equations}, J. Math. Anal. Appl. \textbf{375} (2011), 
412--417.

\bibitem{m}
B.\,Malgrange, \emph{Sur le th\'eor\`eme de Maillet}, Asympt. Anal. \textbf{2}:1 (1989), 1--4.

\bibitem{p} M.\,Petrowitch \emph{Th\`{e}ses: Sur les z\'ero et les infinis des int\'egrales des \'equations diff\'erentielles alg\'ebraiques. 
Propositions donn\'ees par la Facult\'e.} Paris. 1894.

\bibitem{s} 
S.\,Shimomura, \emph{A family of solutions of a nonlinear ordinary differential equation and its application to Painlev\'e equations (III), (V) and 
(VI)}, J. Math. Soc. Japan \textbf{39} (1987), 649--662.

\bibitem{t} 
K.\,Takano, \emph{Reduction for Painlev\'e equations at the fixed singular points of the first kind}, Funkc. Ekvac. \textbf{29} (1986), 99--119.

\bibitem{zh}
C.\,Zhang, \emph{Sur un th\'eor\`eme du type de Maillet--Malgrange pour les \'equations $q$-diff\'erences-diff\'erentielles}, Asympt. Anal. 
\textbf{17}:4 (1998), 309--314.

\end{thebibliography}
\end{document}